\documentclass[11pt]{article}
\usepackage{amsthm,amsfonts, amsbsy, amssymb,amsmath,graphicx}
\usepackage{graphics}
\usepackage[english]{babel}
\usepackage{graphicx}
\usepackage{lineno}
\usepackage{enumerate}
\usepackage{tikz}
\usepackage{hyperref}
\usepackage{color}
\usepackage{tikz}

\hypersetup{colorlinks=true}

\hypersetup{colorlinks=true, linkcolor=blue, citecolor=blue,urlcolor=blue}

\title{Restrained condition on double Roman dominating functions\thanks{The first two authors (Babak Samadi and Nasrin Soltankhah) have been supported by the Iran National Science Foundation (INSF) and Alzarha University, Grant No. 99022321.}\vspace{5mm}}

\author{B. Samadi\thanks{Corresponding author}, N. Soltankhah, H. Abdollahzadeh Ahangar\thanks{Corresponding author}, M. Chellali,\\ D.A. Mojdeh, S.M. Sheikholeslami and J.C. Valenzuela-Tripodoro\\[0.5cm]
}
\date{}

\newtheorem{theorem}{Theorem}[section]
\newtheorem{corollary}[theorem]{Corollary}

\newtheorem{proposition}[theorem]{Proposition}
\newtheorem{obs}[theorem]{Observation}

\theoremstyle{definition}

\theoremstyle{remark}

\begin{document}

\maketitle
\begin{abstract}
We continue the study of restrained double Roman domination in graphs. For a graph $G=\big{(}V(G),E(G)\big{)}$, a double Roman dominating function $f$ is called a restrained double Roman dominating function (RDRD function) if the subgraph induced by $\{v\in V(G)\mid f(v)=0\}$ has no isolated vertices. The restrained double Roman domination number (RDRD number) $\gamma_{rdR}(G)$ is the minimum weight $\sum_{v\in V(G)}f(v)$ taken over all RDRD functions of $G$. 

We first prove that the problem of computing $\gamma_{rdR}$ is NP-hard even for planar graphs, but it is solvable in linear time when restricted to bounded clique-width graphs such as trees, cographs and distance-hereditary graphs. Relationships between $\gamma_{rdR}$ and some well-known parameters such as restrained domination number $\gamma_{r}$, domination number $\gamma$ and restrained Roman domination number $\gamma_{rR}$ are investigated in this paper by bounding $\gamma_{rdR}$ from below and above involving $\gamma_{r}$, $\gamma$ and $\gamma_{rR}$ for general graphs, respectively. We prove that $\gamma_{rdR}(T)\geq n+2$ for any tree $T\neq K_{1,n-1}$ of order $n\geq2$ and characterize the family of all trees attaining the lower bound. The characterization of graphs with small RDRD numbers is given in this paper.   

\noindent \ \ \
\end{abstract}
\vspace{1mm}

%%%%%%%%%%%%%%%%%%%%%%%%%%%%%%%%%%%%%%%%%%%%%%%%%%%%%%%%%%%%%%%%%%%

\section{Introduction and preliminaries}

Throughout this paper, we consider $G$ as a finite simple graph with vertex set $V(G)$ and edge set $E(G)$. We use \cite{w} as a reference for terminology and notation which are not explicitly defined here. The {\em open neighborhood} of a vertex $v$ is denoted by $N(v)$, and its {\em closed neighborhood} is $N[v]=N(v)\cup \{v\}$. The {\em minimum} and {\em maximum degrees} of $G$ are denoted by $\delta(G)$ and $\Delta(G)$, respectively. A {\em universal} vertex of a graph of order $n$ is a vertex of degree $n-1$. For a given subset $S\subseteq V(G)$, by $G[S]$ we represent the subgraph induced by $S$ in $G$. A tree $T$ is a {\em double star} if it contains exactly two vertices that are not leaves. A double star with $p$ and $q$ leaves attached to each support vertex, respectively, is denoted by $S_{p,q}$. For a tree $T$, by $L(T)$ we denote its set of leaves.

A set $S\subseteq V(G)$ is called a {\em dominating set} if every vertex not in $S$ has a neighbor in $S$. The {\em domination number} $\gamma(G)$ of $G$ is the minimum cardinality among all dominating sets of $G$. A {\em restrained dominating set} (RD set) in a graph $G$ is a dominating set $S$ in $G$ for which every vertex in $V(G)\setminus S$ is adjacent to another vertex in $V(G)\setminus S$. The {\em restrained domination number} of $G$, denoted by $\gamma_{r}(G)$, is the smallest cardinality of an RD set of $G$. This concept was formally introduced in \cite{dhhlr} (albeit, it was indirectly introduced in \cite{tp}). In general, the reader can refer to \cite{hhh} for more information on domination parameters.

For a function $f:V(G)\rightarrow\{0,\cdots,k\}$, we let $V^{f}_{i}=\{v\in V(G)\mid f(v)=i\}$ for each $0\leq i\leq k$ (we simply write $V_{i}$ if there is no ambiguity with respect to the function $f$). We call $\omega(f)=f\big{(}V(G)\big{)}=\sum_{v\in V(G)}f(v)$ as the {\em weight} of $f$.

A {\em Roman dominating function} (or an RD function for short) of a graph $G$ is a function $f:V(G)\rightarrow\{0,1,2\}$ such that if $v\in V_0$ for some $v\in V(G)$, then there exists $w\in N(v)$ such that $w\in V_2$. This concept was formally defined by Cockayne \emph{et al.} \cite{cdhh} motivated, in some sense, by the article of Ian Stewart entitled ``Defend the Roman Empire!" (\cite{s}), published in {\it Scientific American}.

A {\em double Roman dominating function} (DRD function for short) of a graph $G$ is a function $ f:V(G)\rightarrow \{0,1,2,3\}$ for which the following conditions are satisfied.

(a) If $f(v)=0$, then $v$ must have at least two neighbors in $V_2$ or one neighbor in $V_3$.

(b) If $f(v)=1$, then the vertex $v$ must have at least one neighbor in $V_2\cup V_3$.\\
The {\em double Roman domination number} $\gamma_{dR}(G)$ equals the minimum weight of a DRD function on $G$. This concept was introduced by Beeler et al. in \cite{bhh}. This parameter was also studied in \cite{Liu} and \cite{ywll}, for instance.

A large number of papers have posed the condition ``$G[V_{0}^{f}]$ has no edges", or equivalently, ``$V_{0}^{f}$ is independent" on the RD functions and DRD functions $f$ of a graph $G$ and introduced the concepts of {\em outer independent Roman domination} and {\em outer independent double Roman domination} in graphs (see, for example, \cite{acs0,acs,mssy}). That ``$G[V_{0}^{f}]$ has no isolated vertices" is another usual condition posed on the RD functions and DRD functions $f$ of a graph $G$. In fact, Pushpam and Padmapriea \cite{pp} initiated the study of restrained Roman dominating functions (RRD functions) in graphs by posing this last condition on the RD functions of the graphs. The restrained Roman domination number $\gamma_{rR}(G)$ is the minimum weight taken over all RRD functions of $G$. This concept was next studied in \cite{jk} and \cite{sas}. Mojdeh et al. \cite{mmv} introduced the concept of restrained double Roman domination in graphs. A {\em restrained double Roman dominating function} (RDRD function) $f$ of $G$ is a DRD function of $G$ for which $G[V_{0}^{f}]$ has no isolated vertices. The {\em double Roman domination number} (RDRD number) $\gamma_{rdR}(G)$ equals the minimum weight taken over all RDRD functions of $G$.

Gao et al. \cite{gxy} have recently showed that the decision problem associated with $\gamma_{rdR}$ is NP-complete for chordal graphs. They also studied this graph parameter for the strong and direct products of two graphs.

We continue the study and investigation of restrained double Roman domination in graphs. This paper is is organized as follows. We first prove that the decision problem associated with the RDRD number is NP-complete even when restricted to planar graphs. In spite of this fact, it is shown in this paper that the problem of computing $\gamma_{rdR}$ can be solved in linear time for the bounded clique-width graphs. We also bound $\gamma_{rdR}$ from below and above for general graphs involving the domination number and the restrained domination number. Section $3$ is dedicated to the study of restrained double Roman domination in trees. In the last section, the characterization of all connected graphs with small RDRD numbers (that is, $\gamma_{rdR}\in\{3,4,5\}$) is given.

For the sake of convenience, by a $\gamma(G)$-set (resp. $\gamma_{r}(G)$-set) we mean a dominating set (resp. RD set) in $G$ of minimum cardinality. Also, by a $\gamma_{rdR}(G)$-function (resp. $\gamma_{rR}(G)$-function) we mean an RDRD function (resp. RRD function) of $G$ with minimum weight.

We close this part by the following properties of RDRD functions.

\begin{obs}
\label{leaf}Let $f=(V_{0},V_{1},V_{2},V_{3})$ be a $\gamma_{rdR}(G)$-function of a graph $G$. Then,
\begin{itemize}
\item[\emph{(}i\emph{)}] Every leaf of $G$ belongs to $V_{1}\cup V_{2}\cup V_{3}$.
\item[\emph{(}ii\emph{)}] If $\left\vert V_{0}\right\vert $ is maximum, then $V_{1}$ is independent, and no vertex in $V_{1}$ has a neighbor in $V_{0}$.
\end{itemize}
\end{obs}
\begin{proof}
We only prove $(ii)$. Assume first that $V_{1}$ is not independent and let $y,z\in V_{1}$ be two adjacent vertices. Let $
y'$ and $z'$ be their respective neighbors in $V_{2}\cup V_{3}$. Then $g(y)=g(z)=0,$ $g(t)=\max \{f(y')+1,f(z')+1,3\}$ for $t\in\{y',z'\}$ and $g(v)=f(v)$ otherwise, is a $\gamma_{rdR}(G)$-function with $|V_{0}^{g}|>|V_{0}|$, a contradiction. Assume now that a vertex $z\in V_{1} $ has a neighbor in $V_{0}.$ By definition, $z$ must have a neighbor $y\in V_{2}\cup V_{3}$. Then the function $g$ defined by $g(z)=0,$ $g(y)=\max\{f(y)+1,3\}$ and $g(t)=f(t)$ otherwise, is a $\gamma_{rdR}(G)$-function for which $|V_{0}^{g}|>|V_{0}|$, a contradiction.
\end{proof}

%%%%%%%%%%%%%%%%%%%%%%%%%%%%%%%%%%%%%%%%%%%%%%%%%%%%%%%%%%%%%%%%%%%%
 
\section{Complexity and computational issues}

We consider the problem of deciding whether a graph $G$ has an RDRD function of weight at most a given integer. That is stated in the following decision problem.

$$\begin{tabular}{|l|}
  \hline
  \mbox{RISTRAINED DOUBLE ROMAN DOMINATION problem (RDRD problem)}\\
  \mbox{INSTANCE: A graph $G$ and an integer $j\leq2|V(G)|$.}\\
  \mbox{QUESTION: Is there an RDRD function $f$ of weight at most $j$?}\\
  \hline
\end{tabular}$$

\subsection{Hardness results}

In what follows, we make use of the ROMAN DOMINATION problem. Schnupp \cite{Schnupp} showed that this problem is NP-complete for planar graphs (see also \cite{hhh} for more explanation).

$$\begin{tabular}{|l|}
  \hline
  \mbox{ROMAN DOMINATION problem}\\
  \mbox{INSTANCE: A graph $G$ and an integer $k\leq|V(G)|$.}\\
  \mbox{QUESTION: Is there an RD function $f$ of weight at most $k$?}\\
  \hline
\end{tabular}$$

Mojdeh et al. \cite{mmv} proved that the RDRD problem is NP-complete for general graphs. They also asked whether this problem for planar graphs belongs to the class of NP-complete problems. In what follows, we answer this question in the affirmative.

\begin{theorem}\label{Comp}
The RDRD problem is NP-complete even when restricted to planar graphs.
\end{theorem}
\begin{proof}
The problem clearly belongs to NP because it can be checked in polynomial time that a given function is indeed an RDRD function of weight at most $j$.

We set $j=4n+k$. For a given graph $G$ with $V(G)=\{v_{1},\cdots,v_{n}\}$, we construct a graph $G'$ as follows. For any $1\leq i\leq n$, let $H_{i}$ be obtained from a copy of $K_{2,4}$ with bipartite sets $\{a_{i},b_{i}\}$ and $\{c_{i},d_{i},e_{i},f_{i}\}$ by adding two edges $c_{i}d_{i}$ and $e_{i}f_{i}$. The graph $G'$ is now obtained by joining $v_{i}$ to both $a_{i}$ and $f_{i}$ for each $1\leq i\leq n$. It is clear from the context that $G'$ is constructed in polynomial time. Moreover, $G'$ is planar if $G$ is.

Let $f$ be a $\gamma_{rdR}(G')$-function for which $|V^{f}_{3}\cap V(G)|$ is minimized. Suppose that $f(v_{j})=3$ for some $1\leq j\leq n$. We observe that $f\big{(}V(H_{j})\big{)}=3$ by the minimality of $f$. In such a situation, $f$ assigns $3$ to $b_{j}$ and $0$ to the other vertices of $H_{j}$. Let $v_{i}$ be any vertex in $N_{G}(v_{j})\cap V^{f}_{0}$. Since $f(v_{i})=0$, it is easy to see that $f\big{(}V(H_{i})\big{)}\geq4$. Taking the minimality of $f$ into account, $f$ assigns $2$ to both $a_{i}$ and $b_{i}$, and $0$ to the other vertices of $H_{i}$. Because $v_{i}$ was an arbitrary vertex in $N_{G}(v_{j})\cap V^{f}_{0}$, it follows that the assignment $\big{(}f'(a_{j}),f'(b_{j}),f'(v_{j})\big{)}=(2,2,2)$ and $f'(v)=f(v)$ for any other vertex $v\in V(G')$ is a $\gamma_{rdR}(G')$-function of $G'$ for which $|V^{f'}_{3}\cap V(G)|<|V^{f}_{3}\cap V(G)|$. This contradicts our choice of $f$. The above argument guarantees that $V^{f}_{3}\cap V(G)=\emptyset$.

We now set $X=\{v\in V^{f}_{0}\cap V(G)\mid \mbox{$v$ does not have any neighbor in}\ V^{f}_{2}\cap V(G)\}$.
Since $f(v_{i})\leq2$ for all $1\leq i\leq n$, it follows that $f\big{(}V(H_{i})\big{)}\geq4$ for all $1\leq i\leq n$. Moreover, we must have $f\big{(}V(H_{i})\big{)}\geq5$ when $v_{i}\in X$ (in this case, $\big{(}f(a_{i}),f(b_{i})\big{)}=(3,2)$ and $f(v)=0$ for the other vertices $v$ of $H_{i}$). We then have

\begin{equation}\label{INE1}
\begin{array}{lcl}
\gamma_{rdR}(G')=\omega(f)&=&\sum_{i=1}^{n}f\big{(}V(H_{i})\big{)}+f\big{(}V(G)\big{)}\\
&=&\sum_{v_{i}\in V(G)\setminus X}f\big{(}V(H_{i})\big{)}+\sum_{v_{i}\in X}f\big{(}V(H_{i})\big{)}+f\big{(}V(G)\big{)}\\
&\geq&4|V(G)\setminus X|+5|X|+f\big{(}V(G)\big{)}\\
&=&4n+|X|+f\big{(}V(G)\big{)}.
\end{array}
\end{equation}

On the other hand, we define $g:V(G)\rightarrow\{0,1,2\}$ as $g(v)=1$ for each $v\in X$, and $g(v)=f(v)$ for any other vertex $v$ of $G$ (this is indeed well-defined as $V^{f}_{3}\cap V(G)=\emptyset$). It is easily seen that $g$ is an RD function of $G$ of weight $\omega(g)\leq|X|+f\big{(}V(G)\big{)}$. Together this inequality and the inequality chain (\ref{INE1}) imply that $\gamma_{rdR}(G')\geq4n+\gamma_{R}(G)$.

Conversely, let $h$ be a $\gamma_{R}(G)$-function. We define $h'$ on $V(G')$ as follows. Let $h'\mid_{V(G)}=h$, $\big{(}h'(a_{i}),h'(b_{i})\big{)}=(2,2)$ for any $1\leq i\leq n$, and $h'(v)=0$ for the other vertices $v$ of $G'$. It is not difficult to see that $h'$ fulfills all conditions of an RDRD function of $G'$. Moreover, $\omega(h')=4n+h\big{(}V(G)\big{)}=4n+\gamma_{R}(G)$. Therefore, $\gamma_{rdR}(G')\leq4n+\gamma_{R}(G)$.

The desired equality $\gamma_{rdR}(G')=4n+\gamma_{R}(G)$ completes our reduction by taking into account the fact that $\gamma_{rdR}(G')\leq j$ if and only if $\gamma_{R}(G)\leq k$. On the other hand, because the RD problem is NP-complete for planar graphs, we have the same with the RDRD problem. This completes the proof.
\end{proof}

However, it is possible to approach the RDRD problem in linear time in some f such as trees family of graphs such as trees and
cographs. In the rest of this part, we prove that the RDRD problem can be solved in linear time for graphs with bounded
clique-width. Among these graphs we can find distance-hereditary graphs, series-parallel graphs and cographs. It is also important to say that any bounded treewidth graph is a bounded clique-width graph. So, we can also consider {the whole of} trees. 

The {\em clique-width} $cw(G)$ of a graph $G$ is the minimum number of labels needed to construct $G$ using the following four
operations:\vspace{0.5mm}\\
$\bullet$ Create a new graph with a single vertex $v$ with label $i$ (this operation is written $i(v)$).\vspace{0.5mm}\\
$\bullet$ Take the disjoint union $G+H$ of two labelled graphs $G$ and $H$.\vspace{0.5mm}\\
$\bullet$ Add an edge between each vertex with label $i$ and each vertex with label $j$, $i\neq j$ (written $\eta_{i,j}$).\vspace{0.5mm}\\
$\bullet$ Relabel every vertex with label $i$ to have label $j$ (written $\rho_{i\rightarrow j}$).\vspace{0.5mm}

A construction of $G$ with the four operations is called $k$-{\em expression} if it uses at most $k$ labels (among $\{1,\ldots,k\}$ for the sake of convenience). Thus the clique-width of $G$ is the minimum $k$ for which $G$ has a $k$-expression. For example, we can construct the star $K_{1,3}$ (on vertices $v_{1}$, $v_{2}$, $v_{3}$ and $v_{4}$ with central vertex $v_{1}$) as $\eta_{1,2}(1({v_1})+(2({v_2})+(2({v_3})+2({v_4}))))$. On the other hand, a construction using only one label is impossible. Therefore, $cw(K_{1,3})=2$.

For our purpose, we make use of a theorem by Courcelle et al. \cite{cmr} stating that every property of a graph $G$, with bounded clique-width, can be decided in linear time whenever the property is definable in the monadic second order logic language and a $k$-expression of $G$ is part of the input. The so-called language of monadic second order logic is the set of formulas formed with the Boolean connectives $\wedge,\vee,\lnot$, the set quantifications $\forall,\exists$ and the objects. Let $G(\tau_{1})$ be the graph $G$ presented as a logic structure $<V(G),\mathcal{R}>$, where $v\ \mathcal{R}\ w=\mathcal{R}(v,w)$ if and only if $vw$ is an edge of $G$. Finally, let us denote by $MSOL(\tau_{1})$ the monadic second order logic with quantifications over the subsets of elements of $G(\tau_{1})$. For more details, the reader can consult \cite{cmr} (and \cite{lklp}, for the restricted definitions).

The optimization problems that can be defined as
\begin{equation}\label{Logic}
Opt \; \left\{ \sum_{1\le i \le l} a_i |X_i| \;:\;<G(\tau_1),X_1,\ldots,X_l> \; \vDash \theta(X_1,\ldots,X_l) \right\}
\end{equation}
where $X_j$ are free set variables, $a_i$ are integers and $\theta$ is a formula belonging to $MSOL(\tau_1)$, form the class of $LinEMSOL(\tau_1)$ optimization problems.

Liedloff et al. \cite{lklp} used the result by Courcelle to show that the complexity of the Roman domination decision problem could be relaxed under certain restrictions on the underlying graphs. We next prove an analogous result regarding the decision problem associated with the RDRD number.

\begin{theorem}
The RDRD problem belongs to $LinEMSOL(\tau_{1})$.
\end{theorem}
\begin{proof}
To prove the result we only have to define the RDRD problem as an optimization problem with the structure
described in (\ref{Logic}).

Let $g=(V_0,V_1,V_2,V_3)$ be a $\gamma_{rdR}(G)$-function. Let us define the set of free set variables $\{X_j\}_{i=0}^{3}$ in which $X_j(v)=1$ if $v\in V_j$ and $X_j(v)=0$ otherwise, for each $0\leq j\leq3$.

We denote by $|X_j|$ the sum $\sum_{v\in V}X_{j}(v)$. Thus, $|X_j|=|V_j|$ for each $0\leq j\leq3$. So, the problem of finding an RDRD function with minimum weight is equivalent to solving the optimization problem
\begin{equation*}
\min_{X_{i}}\left\{|X_{1}|+2|X_{2}|+3|X_{3}|\;:\;<G(\tau_{1}),X_{0},\ldots,X_{3}>\;\vDash \theta(X_{0},\ldots,X_{3})\right\},
\end{equation*}
where $\theta$ is defined as follows:
\begin{equation*}
\begin{array}{c}
\theta (X_{0},\ldots ,X_{3})=\forall v\left[ {\ \vphantom{\frac{\frac{A}{B}}{\frac{A}{B}}}}X_{3}(v)\vee X_{2}(v)\vee \left[ {%
\vphantom{\frac{\int A}{B}}}X_{1}(v)\wedge \exists r\left( \left(
X_{2}(r)\vee X_{3}(r)\right) \wedge \mathcal{R}(v,r)\right) {%
\vphantom{\frac{\int A}{B}}}\right] \right. \\[0.5em]
\vee \left. \left[ {\vphantom{\frac{\int A}{B}}}X_{0}(v)\wedge \left( \
\exists v'\ (X_{0}(v')\wedge \mathcal{R}(v,v')\right) \wedge
\left( {\vphantom{\frac{\int A}{B}}}\exists r\left( X_{3}(r)\wedge \mathcal{R%
}(v,r)\right) \right. \right. \right. \\[0.5em]
\left. \left. \vee \ \exists r,s\left( X_{2}(r)\wedge X_{2}(s)\wedge
\mathcal{R}(v,r)\wedge \mathcal{R}(v,s){\vphantom{\frac{A}{B}}}\right)
\left. {\vphantom{\frac{\int A}{B}}}\right) {\vphantom{\frac{\int A}{B}}}%
\right] {\ \vphantom{\frac{\frac{A}{B}}{\frac{A}{B} }}}\right].
\end{array}
\end{equation*}

We observe that $\theta$ checks that the function is RDRD. The first three sub-clauses verify what happens when the label of the vertex $v$ is $3,2$ or $1$, while the last one checks that if the vertex $v$ has the label $0$, then the conditions for $g$ to be a DRD function are fulfilled, as well as there must be another vertex labeled with $0$ belonging to its neighborhood. Hence, $g$ is an RDRD function if and only if the expression $\theta$ is satisfied. This completes the proof.
\end{proof} 

\begin{corollary}
The RDRD problem can be solved in linear time for $k$-bounded clique-width graphs whenever either there exists a linear-time
algorithm to make a $k$-expression of the given graph or a $k$-expression is part of the input.
\end{corollary}

As bounded treewidth graphs are also bounded clique-width, the former corollary can be applied to this family of graphs. In particular, the problem of finding a $\gamma_{rdR}(T)$-function for a tree $T$ can be solved in linear time as it is well-known that every tree has clique-width at most three.

\subsection{Bounds}

As a consequence of Theorem \ref{Comp}, we conclude that the problem of computing the RDRD number even when restricted to planar graphs is NP-hard. Thus, it would be desirable to bound the RDRD number in terms of several invariants of the graph.

\begin{theorem}\label{Rest}
For any graph $G$ of order $n$ with maximum degree $\Delta$,
$$\gamma_{rdR}(G)\geq\frac{2n+(\Delta-2)\gamma_{r}(G)}{\Delta}.$$
This bound is sharp.
\end{theorem}
\begin{proof}
Let $f=(V_{0},V_{1},V_{2},V_{3})$ be a $\gamma_{rdR}(G)$-function. Let $A=\{v\in V_{0}\mid \mbox{$v$ has a neighbor in $V_{3}$}\}$ and $A'=V_{0}\setminus A$. Because every vertex in $V_{3}$ has at most $\Delta$ neighbors in $A$ and every vertex in $A$ has at least one neighbor in $V_3$, it follows that $|A|\leq \Delta|V_{3}|$. On the other hand, every vertex in $A'$ is adjacent to at least two vertices in $V_{2}$. Moreover, each vertex of $V_{2}$ has at most $\Delta$ neighbors in $A'$. Together these two statements imply that $2|A'|\leq \Delta|V_{2}|$. Using now the last two inequalities, we get $2|V_{0}|=2|A|+2|A'|\leq \Delta(|V_{2}|+2|V_{3}|)$.

Since $\gamma_{rdR}(G)=\omega(f)=|V_{1}|+2|V_{2}|+3|V_{3}|$ and $|V_{0}|+|V_{1}|+|V_{2}|+|V_{3}|=n$, we have
\begin{equation}\label{INE2}
\begin{array}{lcl}
\Delta \gamma_{rdR}(G)&=&\Delta(|V_{1}|+2|V_{2}|+3|V_{3}|)=\Delta(|V_{1}|+|V_{2}|+|V_{3}|)+\Delta(|V_{2}|+2|V_{3}|)\\
&\geq& \Delta(|V_{1}|+|V_{2}|+|V_{3}|)+2|V_{0}|=(\Delta-2)(|V_{1}|+|V_{2}|+|V_{3}|)+2n.
\end{array}
\end{equation}

On the other hand, $V_{1}\cup V_{2}\cup V_{3}$ is an RD set in $G$. Therefore, $|V_{1}|+|V_{2}|+|V_{3}|\geq \gamma_{r}(G)$. Together this and the inequality (\ref{INE2}) lead to the desired lower bound on $\gamma_{rdR}(G)$.

That the bound is sharp, may be seen as follows. Let $s$, $p$ and $q$ be three positive integers, where $s\geq4$. We consider the cycle $C_{t}:v_{1}v_{2}\cdots v_{t}v_{1}$ on $t=(p+q)s$ vertices. Let $H$ be the graph obtained from $C_{t}$ by adding the new vertices $x_{1},\cdots,x_{p},y_{1},\cdots,y_{q},z_{1},\cdots,z_{q}$, and setting $N_{H}(x_{i})=\{v_{(i-1)s+1},\cdots,v_{is}\}$ and $N_{H}(y_{j})=N_{H}(z_{j})=\{v_{(p+j-1)s+1},\cdots,v_{(p+j)s}\}$ for each $1\leq i\leq p$ and $1\leq j\leq q$. It is easily checked that $h(x_{i})=3$ and $h(y_{j})=h(z_{j})=2$ for all $1\leq i\leq p$ and $1\leq j\leq q$, and $h(v)=0$ for any other vertex $v$ is an RDRD function of $H$ with weight $3p+4q$. Therefore, $\gamma_{rdR}(H)\leq3p+4q$. Using $|V(H)|=(p+q)s+p+2q$, $\Delta(H)=s$, $\gamma_{r}(H)=p+2q$ and applying the lower bound on $\gamma_{rdR}(H)$, we have $\gamma_{rdR}(H)\geq3p+4q$. Therefore, the lower bound gives the exact value for $\gamma_{rdR}(H)$.
\end{proof}

There is also a simple but strong relationship between the RDRD number and the domination number of a graph.

\begin{proposition}
For any graph $G$ of order $n$, $\gamma_{rdR}(G)\leq n+\gamma(G)$. This bound is sharp. 
\end{proposition}
\begin{proof}
For any $\gamma(G)$-set $S$, we define $f:V(G)\rightarrow\{0,1,2,3\}$ by $f(u)=2$ and $f(v)=1$ for each $u\in S$ and $v\in V(G)\setminus S$. It is easily observed that $f$ is an RDRD function of $G$. So, $\gamma_{rdR}(G)\leq \omega(f)=n+|S|=n+\gamma(G)$.

For any graph $H$ with no isolated vertices on the set of vertices $\{v_{1},\cdots,v_{n}\}$, let $H'$ be obtained from $H$ by joining two new vertices $v_{i1}$ and $v_{i2}$ to $v_{i}$ for each $1\leq i\leq n$. It is then easy to see that $f'(x)=0$ and $f'(y)=2$, for each $x\in V(H)$ and $y\in V(H')\setminus V(H)$, is a $\gamma_{rdR}(H')$-function. Hence, $\gamma_{rdR}(H')=4|V(H)|=|V(H')|+\gamma(H')$. So, the upper bound is sharp.   
\end{proof}

For regular graphs of girth greater than or equal to six, we show the following upper bound.

\begin{proposition}\label{regul}
Let $r\geq 3$ be an integer and let $G$ be an $r$-regular graph of order $n$ with girth at least six. Then, $\gamma_{rdR}(G)\leq 2(n-r^{2})+1$.
\end{proposition}
\begin{proof}
Let $y$ be a vertex of $V(G)$ and consider the function $f$ defined by $f(y)=3$, $f(t)=0$ for all $t\in N(y)\cup(N\big{(}N(y)\big{)}\setminus\{y\})$ and $f(v)=2$ for any other vertex $v$. We observe that any vertex in $N(y)\subseteq V_{0}$ is adjacent to $y\in V_{3}$ and it is also adjacent to at least one vertex in $N\big{(}N(y)\big{)}\setminus\{y\}\subseteq V_{0}$. In addition, since $r\geq 3$ and because each cycle in $G$ has length greater than or equal to six, any vertex in $N\big{(}N(y)\big{)}\setminus\{y\}\subseteq V_{0}$ is adjacent to at least two vertices in $V_{2}$. Hence, $f$ is an RDRD function in $G$ and we have that
$$\gamma_{rdR}(G)\leq \omega(f)=2\big{(}n-(r+1)-r(r-1)\big{)}+3=2n-2r^{2}+1.$$
\end{proof}

To see that the bound is tight, we can consider the Heawood graph $H$ (depicted in Figure \ref{Heawood}) with $14$ vertices and girth six. It is readily seen that $\gamma_{r}(H)=4$. We then, by Theorem \ref{Rest}, have $\gamma_{rdR}(H)\geq 11$. Now, by taking Proposition \ref{regul} into account, we deduce that $\gamma_{rdR}(H)=11$.

It is important to note that the requirement about the girth of the graph in Proposition \ref{regul} cannot be relaxed. To see this, we consider the Petersen graph $G$ with girth $5$. Using the upper bound given in Proposition \ref{regul}, we have $\gamma_{rdR}(G)\leq 3$, which is impossible.

\begin{figure}[h]
\tikzstyle{every node}=[circle, draw, fill=white!, inner sep=0pt,minimum width=.16cm]
\begin{center}
\begin{tikzpicture}[thick,scale=.6]
  \draw(0,0) { % <-- START CO-ORDINATES

+(-0.75,0) node{}
+(-2,-0.65) node{}
+(-2.9,-1.7) node{}
+(-3.3,-2.9) node{}
+(-2.9,-4.1) node{}
+(-2,-5.2) node{}
+(-0.75,-6) node{}

+(0.75,0) node{}
+(2,-0.65) node{}
+(2.9,-1.7) node{}
+(3.3,-2.9) node{}
+(2.9,-4.1) node{}
+(2,-5.2) node{}
+(0.75,-6) node{}

+(-0.75,0) -- +(-2,-0.65) -- +(-2.9,-1.7) -- +(-3.3,-2.9) -- +(-2.9,-4.1) -- +(-2,-5.2) -- +(-0.75,-6) -- +(0.75,-6) -- +(2,-5.2) -- +(2.9,-4.1) -- +(3.3,-2.9) -- +(2.9,-1.7) -- +(2,-0.65) -- +(0.75,0) -- +(-0.75,0)

+(0.75,0) -- +(-2.9,-4.1)
+(2,-0.65) -- +(0.75,-6)
+(2.9,-1.7) -- +(-2.9,-1.7)
+(3.3,-2.9) -- +(-2,-5.2)
+(2.9,-4.1) -- +(-0.75,0)
+(2,-5.2) -- +(-3.3,-2.9)
+(-0.75,-6) -- +(-2,-0.65)

};
\end{tikzpicture}
\end{center}
\caption{The Heawoon graph $H$}\label{Heawood}
\end{figure}
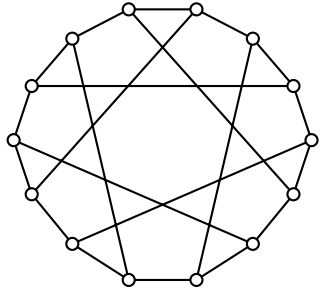 

It has been shown in \cite{mmv} that for every graph $G$, $\gamma_{rdR}(G)\leq 2\gamma_{rR}(G)$ with equality if and only if $G\cong \overline{K_{p}}$. Consequently, if $G$ is a nontrivial graph, then $\gamma_{rdR}(G)\leq 2\gamma_{rR}(G)-1$. Our next result shows that $K_{2}$ is the only nontrivial connected triangle-free graph satisfying $\gamma_{rdR}(G)=2\gamma_{rR}(G)-1$.

\begin{proposition}\label{free}
If $G$ is a connected triangle-free graph on $n\geq3$ vertices, then $\gamma_{rdR}(G)\leq 2\gamma_{rR}(G)-2$.
\end{proposition}
\begin{proof}
Assume that $\gamma_{rdR}(G)=2\gamma_{rR}(G)-1$. Among the whole of $\gamma_{rR}(G)$-functions, let $f=(V_{0},V_{1},V_{2})$ be the one for which $|V_{2}|$ is maximized. Since $(V_{0},\emptyset,V_{1},V_{2})$ is an RDRD function of $G$, it follows that $2\gamma_{rR}(G)-1=\gamma_{rdR}(G)\leq 2|V_{1}|+3|V_{2}|$. Using the fact that $\gamma_{rR}(G)=|V_{1}|+2|V_{2}|$, we deduce that $|V_{2}|\leq 1$. If $V_{2}=\emptyset $, we must have $V_{0}=\emptyset$ and $V_{1}=V(G)$. By our choice of $f$ and the connectedness of $G$, we conclude that $n\leq 3$ (and thus $n=3$) or $G$ is a star. Since $G$ is triangle-free, we have $G\cong K_{1,n-1}$. But then $\gamma_{rdR}(K_{1,n-1})\neq 2\gamma_{rR}(K_{1,n-1})-1$, a contradiction. Hence, we have $|V_{2}|=1$ with $V_{2}=\{x\}$. If $V_{0}=\emptyset $, then relabeling $x$ with the value $1$ instead of $2$ gives us an RRD function with weight $\gamma_{rR}(G)-1$, a contradiction. Hence $V_{0}\neq \emptyset$, and thus $x$ is adjacent to all vertices in $V_{0}$. But then $x$ along with any two adjacent vertices from $V_0$ induce a triangle, a contradiction. So, the desired upper bound holds.
\end{proof}

Note that the condition ``triangle-free" is necessary in Proposition \ref{free}. For instance, the upper bound does not hold for $K_{1} \vee pK_{2}$, in which $p\geq1$.

We now relate the RDRD number to the domination and restrained domination numbers for each graph. Recall that
a subset $S\subseteq V(G)$ is a {\em $2$-dominating set} if each vertex of $V(G)\setminus S$ has at least two neighbors in $S$. A $2$-dominating set $S$ of $G$ with the additional property that the induced subgraph $G[V(G)\setminus S]$ has no isolated vertices is called a {\em restrained 2-dominating set} (R$2$D set for short) of $G$. The (resp. {\em restrained}) $2$-{\em domination number} \big{(}resp. $\gamma_{r2}(G)$\big{)} $\gamma_{2}(G)$ equals the minimum cardinality of a (resp. restrained) $2$-dominating set in $G$.

\begin{theorem}\label{frame}
For every connected graph $G$,
$$\gamma_{rdR}(G)\geq \gamma(G)+\gamma_{r}(G).$$
Moreover, the equality holds if and only if $G$ is a star or $\gamma_{r2}(G)=\gamma_{r}(G)=\gamma(G)$.
\end{theorem}
\begin{proof}
Let $f=(V_{0},V_{1},V_{2},V_{3})$ be a $\gamma_{rdR}(G)$-function. Obviously, $V(G)\setminus V_{0}$ is an RD set and $V_{2}\cup V_{3}$ is a dominating set of $G$. Because $\gamma_{rdR}(G)=|V_{1}|+2|V_{2}|+3|V_{3}|,$ we deduce that
\begin{equation}\label{aa}
\gamma_{rdR}(G)=(|V_{1}|+|V_{2}|+|V_{3}|)+(|V_{2}|+|V_{3}|)+|V_{3}|\geq \gamma_{r}(G)+\gamma(G)+|V_{3}|\geq \gamma_{r}(G)+\gamma(G),  
\end{equation}
and the lower bound follows.

Clearly, the equality holds if $G$ is isomorphic to a star. Suppose now that $\gamma_{r2}(G)=\gamma_{r}(G)=\gamma(G)$. For any $\gamma_{r2}(G)$-set $S$, the function $f(x)=2$ for $x\in S$ and $f(y)=0$ for each $y\in V(G)\setminus S$ defines an RDRD function of $G$. So, $\gamma_{rdR}(G)\leq \omega(f)=2|S|=2\gamma_{r2}(G)=\gamma_{r}(G)+\gamma(G)$. By (\ref{aa}), we get $\gamma_{rdR}(G)=\gamma(G)+\gamma_{r}(G)$, as desired.

Conversely, assume that $\gamma_{rdR}(G)=\gamma(G)+\gamma_{r}(G)$. We deduce from (\ref{aa}) that $V_{3}=\emptyset$, $|V_{1}|+|V_{2}|=\gamma_{r}(G)$ and $|V_{2}|=\gamma(G)$. Hence, we may assume that for any $\gamma_{rdR}(G)$-function $f$, no vertex is labeled with $3$ under $f$. We proceed with some claims.\vspace{1mm}\\
Claim $1$. \textit{No edge of $G$ joins $V_{1}$ and $V_{0}$.}\\
\textit{Proof of Claim $1$.} Suppose that there is an edge $uv\in E(G)$ for which $u\in V_{1}$ and $v\in V_{0}$. By definition, $u$ has a neighbor $w\in V_{2}$. Then, the function $g$ defined on $V(G)$ by $g(w)=3$, $g(u)=0$ and $g(x)=f(x)$ otherwise, is a $\gamma_{rdR}(G)$-function, which contradicts our earlier assumption. $(\square)$\vspace{1mm}\\
Claim $2$. \textit{$V_{1}$ is an independent set.}\\
\textit{Proof of Claim $2$.} Suppose that two vertices $u,v\in V_{1}$ are
adjacent. By definition, $u,v$ have neighbors $u',v'\in V_{2}$, respectively. Then, the assignment $g(u')=g(v')=3$, $g(u)=g(v)=0$ and $g(x)=f(x)$ for any other vertex $x$ gives a $\gamma_{rdR}(G)$-function, which contradicts our earlier
assumption. $(\square)$\vspace{1mm}\\
Claim $3$. \textit{If some vertex $u$ in $V_{2}$ has a neighbor in $V_{0}$, then $N(u)\subseteq V_{0}$.}\\
\textit{Proof of Claim $3$.} Suppose that there exists a vertex $u$ in $V_{2}$ such that $N_{G}(u)\cap V_{0}\neq \emptyset$ and $N(u)\not\subseteq V_{0}$. Since each vertex in $V_{0}$ has at least two neighbors in $V_{2}$, it follows that $(V_{1}\cup V_{2})\setminus \{u\}$ is an RD set of $G$, which contradicts the fact that $|V_{1}|+|V_{2}|=\gamma_{r}(G)$. $(\square)$\vspace{1mm}

In the sequel, we let $V_{2}^{0}$ denote the set of vertices of $V_{2}$ having a neighbor in $V_{0}$ and let $V_{2}^{1}=V_{2}\setminus V_{2}^{0}$. Let $H$ and $K$ be the subgraphs of $G$ induced by $V_{1}\cup V_{2}^{1}$ and $V_{0}\cup V_{2}^{0}$, respectively. Trivially, $V(G)=V(H)\cup V(K)$ and by Claims $1$ and $3$, there are no edges between $V(H)$ and $V(K)$. It follows from the connectedness of $G$ that either $V_{0}\cup V_{2}^{0}=\emptyset$ or $V_{1}\cup V_{2}^{1}=\emptyset$. We now consider these two cases.\vspace{0.5mm}

\textit{Case 1.} $V_{0}\cup V_{2}^{0}\neq \emptyset$. Thus, $V_{1}\cup V_{2}^{1}=\emptyset$ and evidently we have $V_{0}\neq
\emptyset$. Since any vertex in $V_{0}$ must be adjacent to at least two vertices in $V_{2}$, we have $|V_{2}|\geq 2$ as well as $V_{2}$ is an R$2$D set of $G$. Thus, $\gamma_{r2}(G)\leq|V_{2}|=\gamma_{r}(G)=\gamma(G)$. On the other hand, since each R$2$D set of $G$ is an RD set of $G$, we get $\gamma_{r2}(G)\geq \gamma_{r}(G)=\gamma(G)$. Hence, $\gamma_{r2}(G)=\gamma_{r}(G)=\gamma (G)$.\vspace{0.5mm}

\textit{Case 2.} $V_{1}\cup V_{2}^{1}\neq \emptyset$. Thus $V_{0}\cup V_{2}^{0}=\emptyset$, and so $\gamma_{r}(G)=|V_{1}|+|V_{2}|=|V(G)|$. If there are two adjacent vertices $u$ and $v$ each of degree at least two, then $V(G)\setminus \{u,v\}$ is an RD set of $G$, a contradiction with $\gamma_{r}(G)=|V(G)|$. Thus, at least one of the end-points of any edge is leaf and this implies that $G$ is a star. This finishes the proof.
\end{proof}

We next explicitly characterize the connected regular claw-free graphs $G$ with $\gamma_{rdR}(G)=\gamma_{r}(G)+\gamma(G)$. To this aim, we need to recall some more definitions and a result due to Hansberg et al. \cite{hrv}. For an even integer $n\geq 4$, let $H_{n}$ be the $(n-2)$-regular graph of order $n$, or equivalently, $H_{n}$ is obtained from the complete graph $K_{n}$ by removing a perfect matching. It was shown in \cite{hrv} that the only connected regular claw-free graphs $G$ satisfying $\gamma_{2}(G)=\gamma(G)$ are the graph $H_{n}$ and the Cartesian product of two complete graphs of the same order. It is worth noting that the characterization of graphs $G$ for which $\gamma(G)=\gamma_{2}(G)$ remains open and it is only known in some particular classes of graphs.

\begin{proposition}
Let $G$ be a connected regular claw-free graph of order $n$. Then, $\gamma_{rdR}(G)=\gamma(G)+\gamma_{r}(G)$ if and only if
either $G\in \{P_{1},P_{2},P_{3}\}$, $G=H_{n}$ with $n\geq 6$, or $G=K_{p}\square K_{p}$ with $p\geq 3$ and $p^{2}=n$.
\end{proposition}
\begin{proof}
Let $G$ be a connected regular claw-free graph of order $n\geq2$ for which $\gamma_{rdR}(G)=\gamma(G)+\gamma_{r}(G)$. By Theorem \ref{frame}, $G$ is a star or $\gamma_{r2}(G)=\gamma_{r}(G)=\gamma(G)$. If $\gamma_{r2}(G)=\gamma_{r}(G)=\gamma(G)$, then $\gamma(G)=\gamma_{2}(G)$ as $\gamma(G)\leq \gamma_{2}(G)\leq \gamma_{r2}(G)=\gamma(G)$. It follows that $G=H_{n}$ or $G=K_{p}\square K_{p}$ with $p\geq 2$ and $p^{2}=n$. Note that $H_{4}=K_{2}\square K_{2}=C_{4}$ will be excluded because $\gamma_{rdR}(C_{4})=6\neq \gamma(C_{4})=\gamma_{r}(C_{4})$. Therefore, $n\geq6$ and $p\geq3$. If $G$ is a star, we necessarily have $G\in \{P_{2},P_{3}\}$ as $G$ is a claw-free graph.

Conversely, assume that $G\in \{P_{1},P_{2},P_{3}\}$, $G=H_{n}$ with $n\geq 6$ or $G=K_{p}\square K_{p}$ with $p\geq 3$ and $p^{2}=n$. Definitely, if $G\in \{P_{1},P_{2},P_{3}\}$, then $\gamma_{rdR}(G)=\gamma(G)+\gamma_{r}(G)$. Let $G=H_{n}$ with $n\geq 6$. Then, $\gamma(H_{n})=\gamma_{r}(H_{n})=2$ (for instance, any two non-adjacent veetices of $H_{n}$ form both a
minimum dominating set and a minimum RD set of $H_{n}$). In addition, assigning $2$ to any two non-adjacent vertices of $H_{n}$ and $0$ to the other vertices provides an RDRD function of $G$. Therefore, $4=\gamma(G)+\gamma_{r}(G)\leq \gamma _{rdR}(H_{n})\leq 4$, and the equality follows. Finally, let $G=K_{p}\square K_{p}$ with $p\geq 3$ and $p^{2}=n$. Clearly, $\gamma(G)=\gamma_{r}(G)=p$. Also, if $V(K_{p})=\{x_{1},...,x_{p}\}$, then assigning $2$ to each of the $p$ vertices of $G$ of the type $(x_{i},x_{i})$ and $0$ to the remaining vertices gives an RDRD function of $G$. Hence, $2p=\gamma(G)+\gamma_{r}(G)\leq \gamma_{rdR}(G)\leq 2p$, and the equality follows.
\end{proof}

We now turn our attention to the family of all trees $T$ for which $\gamma_{rdR}(T)=\gamma_{r}(T)+\gamma (T)$. For this purpose, we need the result of Dankelmann et al. \cite{dhhs} concerning the characterization of trees with equal domination and restrained domination numbers. They defined a {\em labeling} of a tree $T$ as a mapping $S:V(T)\rightarrow \{X,Y\}$. The label of
a vertex $v$ is also called its {\em status}, denoted $sta(v)$. The sets of vertices of $X$ and $Y$ are denoted $S_{X}(T)$ and $S_{Y}(T)$, respectively. A labeled $K_{1}$ means a $K_{1}$ whose vertex is labeled with the status $Y$. Let $\mathcal{T}$ be the family of trees that can be labeled so that the resulting family of labeled trees contain a labeled $K_{1}$ and is closed under the operations $\mathcal{O}_{1}$ and $\mathcal{O}_{2}$ listed below.\vspace{1mm}\\
\textit{Operation $\mathcal{O}_1$:} Attach to a vertex $v$ of status $X$ a path $vxy$ where $sta(x)=X$ and $sta(y)=Y$.\vspace{1mm}\\
\textit{Operation $\mathcal{O}_2$:} Attach to a vertex $v$ of status $Y$ a path $vxyz$ where $sta(x)=sta(y)= X$ and $sta(z)=Y$.\vspace{1mm}

It has been shown in \cite{dhhs} that\\
$(i)$ every $v\in S_X$ is adjacent to at least one vertex in $S_X$ and to precisely one vertex in $S_Y$,\\
$(ii)$ $S_Y$ is the unique $\gamma_r(T)$-set, and\\
$(iii)$ for any tree $T$, $\gamma(T)=\gamma_r(T )$ if and only if $T\in \mathcal{T}$.

\begin{proposition}\label{trees}
Let $T$ be a tree of order $n$. Then, $\gamma_{rdR}(T)=\gamma (T)+\gamma_{r}(T)$ if and only is $T$ is a star.
\end{proposition}
\begin{proof}
Let $T$ be a tree with $\gamma_{rdR}(T)=\gamma(T)+\gamma_{r}(T)$. If $T$ is a star, then we are done. Suppose that $T$ is
not a star and let $f=(V_{0},V_{1},V_{2},V_{3})$ be a $\gamma_{rdR}(T)$-function. Following the same argument as used to prove the lower bound of Theorem \ref{frame} along with Theorem \ref{frame} itself, we can easily see that $V_{3}=V_{1}=\emptyset$ and $V_{2}$ is a $\gamma_{r}(T)$-set of $T$. Hence, $V_{2}=S_{Y}(T)$. In such a situation, $f$ cannot be an RDRD function of $T$ because each vertex labeled $X$ has precisely one neighbor in $V_{2}$. This contradiction shows that $T$ is a star.
\end{proof}

As an immediate result from Proposition \ref{trees}, we have the following corollary.

\begin{corollary}\emph{(\cite{mmv})}
For any tree $T$, $\gamma_{rdR}(T)=\gamma_{r}(T)+1$ if and only if $T$ is a star.
\end{corollary}

We deduce from Proposition \ref{trees} that for every tree $T$ of diameter at least three satisfies $\gamma _{dR}^{r}(T)\geq \gamma(T)+\gamma _{r}(T)+1$. By the way, characterizing trees $T $ with $\gamma_{rdR}(T)=\gamma(T)+\gamma _{r}(T)+1$ is an interesting problem.

%%%%%%%%%%%%%%%%%%%%%%%%%%%%%%%%%%%%%%%%%%%%%%%%%%%%%%%%%%%%%%%%%%%%%%%%%%%%%%%%%%%%%%

\section{Trees}

We will see in this section that the star is the only nontrivial tree having the smallest RDRD number among all nontrivial trees of the same order. More formally, we prove the following result. We then characterize all extremal trees attaining the lower bound given in the following theorem.

\begin{theorem}\label{Tree2}
For any tree $T$ of order $n\geq2$ different from the star $K_{1,n-1}$, $\gamma_{rdR}(T)\geq n+2$. %This bound is sharp.
\end{theorem}
\begin{proof}
We proceed by induction on the order $n$ of $T$. The result trivially holds for $n=2$. Suppose that $\gamma_{rdR}(T')\geq n'+2$ for any tree $T'\neq K_{1,n'-1}$ of order $2\leq n'<n$. Let $T\neq K_{1,n-1}$ be a tree of order $n$. In particular, this implies that $\mbox{diam}(T)\geq3$.

Let $f$ be a $\gamma_{rdR}(T)$-function. Suppose first that $T$ has a strong support vertex $u$. Let $L_{u}$ be the set of leaves adjacent to $u$. Clearly, $f(x)\geq1$ for all $x\in L_{u}$. Suppose that $f(x)=1$ for some $x\in L_{u}$. We set $T'=T-x$. Notice that $T'\neq K_{1,n-2}$, otherwise $T=K_{1,n-1}$ or $T$ is obtained from a star by joining the vertex $x$ to one of its leaves (which is necessarily the vertex $u$). The first statement is a contradiction, and the second statement is impossible because $u$ is a strong support vertex. Therefore, $\gamma_{rdR}(T')\geq n+1$ by the induction hypothesis. On the other hand, it is clear that $f'=f\mid_{V(T')}$ is an RDRD function of $T'$. We than have
$$\gamma_{rdR}(T)-1=\omega(f')\geq \gamma_{rdR}(T')\geq(n-1)+2,$$
and so $\gamma_{rdR}(T)\geq n+2$.

Suppose now that $f(x)\geq2$ for each $x\in L_{u}$. Since $f$ is a $\gamma_{rdR}(T)$-function, it follows that $f(x)=2$ for all $x\in L_{u}$. Let $x$ and $x'$ be two distinct leaves adjacent to $u$. Then, $f'(x')=3$ and $f'(y)=f(y)$ for all $y\in V(T)\setminus\{x,x'\}$ is an RDRD function of $T'=T-x$ of weight $\omega(f)-1$. Therefore,
$$\gamma_{rdR}(T)-1=\omega(f')\geq \gamma_{rdR}(T')\geq(n-1)+2.$$
This results in the desired lower bound.

From now on, we may assume that $T$ has no strong support vertices. Let $r$ and $v$ be two leaves with $d_{T}(r,v)=\mbox{diam}(T)$. We root the tree $T$ at $r$. Let $u$ be the parent of $v$, and $w$ be the parent of $u$. Suppose that $T'=T-v=K_{1,n-2}$. Since $T$ is different from a star, it follows that $T$ is obtained from $T'$ by joining $v$ to a leaf of $T'$. It is then easy to see that $\gamma_{rdR}(T)=n+2$. So, we assume that $T'\neq K_{1,n-2}$. If $f(v)=1$, then we get the lower bound similar to the argument given in the second paragraph of the proof. So, let $f(v)\geq2$. We now distinguish two cases depending on $f(v)$.\vspace{0.5mm}

\textit{Case 1.} $f(v)=2$. Because $f$ is a $\gamma_{rdR}(T)$-function, it follows that $f(u)\leq1$. Note that if $f(u)=0$, then $f(w)=0$ since $T[V^{f}_{0}]$ has no isolated vertices. Hence, $f\big{(}N_{T}(u)\big{)}<3$, a contradiction. Therefore, $f(u)=1$. We observe that the assignment $\big{(}g(v),g(u)\big{)}=(1,2)$ and $g(y)=f(v)$ for any other vertex $y$ defines a $\gamma_{rdR}(T)$-function such that $g(v)=1$. Since $T'\neq K_{1,n-2}$, we have again the lower bound similarly to the second paragraph of the proof.\vspace{0.5mm}

\textit{Case 2.} $f(v)=3$. Since $f$ is a $\gamma_{rdR}(T)$-function, we have $f(u)=0$. This in particular implies that $f(w)=0$. We consider two more possibilities depending on the behavior of $w$.\vspace{0.25mm}

\textit{Subcase 2.1.} $w$ is adjacent to a vertex labeled $0$ under $f$ different from $u$. Let $T''=T-v-u$. Since $\mbox{diam}(T)\geq3$, it follows that $|V(T'')|=n-2\geq2$. Suppose now that $T''=K_{1,n-3}$. It is not difficult to see that $T$ is obtained from $T''$ by adding the two vertices $u$ and $v$, and the edges $uv$ and $uw$ such that either\vspace{0.25mm}

$(i)$ $w$ is a leaf of $T''$, or

$(ii)$ $w$ is the central vertex of $T''$.\vspace{0.25mm}\\
In both possible cases, we observe that $\gamma_{rdR}(T)=n+2$. Therefore, we assume that $T''\neq K_{1,n-3}$. It is easily observed that $f''=f\mid_{V(T'')}$ is an RDRD function of $T''$. We then have
$$\gamma_{rdR}(T)-3=\omega(f'')\geq \gamma_{rdR}(T'')\geq(n-2)+2,$$
and so $\gamma_{rdR}(T)>n+2$.\vspace{0.25mm}

\textit{Subcase 2.2.} $w$ has no neighbor labeled $0$ under $f$ different from $u$. In such a situation, the assignment $h(w)=1$ and $h(y)=f(y)$ for the other vertices $y\in V(T)\setminus\{u,v,w\}$ gives us an RDRD function of $T''$. Thus,
$$\gamma_{rdR}(T)-2=\omega(h)\geq \gamma_{rdR}(T'')\geq(n-2)+2.$$
Therefore, $\gamma_{rdR}(T)\geq n+2$.\vspace{0.25mm}

All in all, we have proved the desired lower bound for all non-star trees of order at least two.
\end{proof}

As an immediate consequence of Theorem \ref{Tree2}, we have the following result.

\begin{corollary}\label{cor-tree}
For any tree $T$ of order $n\geq2$, $\gamma_{rdR}(T)\geq n+1$ with equality if and only if $T$ is isomorphic to a star.
\end{corollary}

In what follows, we give the characterization of all trees for which the lower bound given in Theorem \ref{Tree2} holds with equality. For this purpose, we introduce the following families of trees.

\begin{itemize}
\item Let $\mathcal{T}_1$ be the family of trees $T$ obtained from a double star with support vertices $u$ and $v$ by subdividing the edge $uv$ at most twice.  

\item Let $T$ be a tree whose support vertices are adjacent to at most two leaves and the distance between any two leaves with different support vertices is $k\equiv0$ $(\mbox{mod}\ 3)$. We then set $\mathcal{T}_2$ to be the family of trees obtained from the trees $T$ by joining $r\geq0$ new vertices to the vertices at distance $k\equiv0$ $(\mbox{mod}\ 3)$ from any leaf $x$ (and the leaf at distance two from $x$ if any).
\end{itemize}

Obviously, any tree attaining the lower bound given in Theorem \ref{Tree2} is of order at least four.

\begin{theorem}\label{Tree1-1}
For any tree $T$ of order $n\geq4$, $\gamma_{rdR}(T)=n+2$ if and if only $T\in \mathcal{T}_1\cup \mathcal{T}_2$.
\end{theorem}
\begin{proof}
Let $T$ be a tree of order $n\geq4$ for which $\gamma_{rdR}(T)=n+2$ and let $f=(V_{0},V_{1},V_{2},V_{3})$ be a $\gamma_{rdR}(T)$-function.

Let $V_0=\emptyset$. Since $|V_1|+|V_2|+|V_3|=n$ and $|V_1|+2|V_2|+3|V_3|=n+2$, it follows that $|V_2|+2|V_3|=2$. If $V_3\neq \emptyset$, it necessarily results in $(|V_2|,|V_3|)=(0,1)$. So, $T$ has a vertex $x$ labeled with $3$ under $f$ and the other vertices of $T$ are labeled $1$ under $f$ adjacent to $x$. This is impossible as $T\ncong K_{1,n-1}$. Therefore, $(|V_2|,|V_3|)=(2,0)$. Let $V_2=\{u,v\}$. If $uv\in E(T)$, it is easily seen that $T$ is isomorphic to a double star. Thus, $T\in \mathcal{T}_1$. So, let $uv\notin E(T)$. Let $P$ be the unique $u,v$-path in $T$. Since every vertex in $V_1$ must be adjacent to $u$ or $v$, this shows that the length $\ell(P)$ of $P$ is at most three. We recall that $T$ is different from a star and that $n\geq4$. It is not difficult to see that in both possibilities $\ell(P)=2$ and $\ell(P)=3$, $T$ belongs to $\mathcal{T}_1$. So, from now on, we may assume that $V_0\neq \emptyset$.\vspace{0.5mm}

Suppose to the contrary that there exists a vertex $w\in V_0$ adjacent to two vertices $u\in V_3$ and $v\in V_2\cup V_3$. By removal the edge $wv$, we obtain two components $T_1$ and $T_2$ with $\gamma_{rdR}(T_1)\geq|V(T_1)|+1$ and $\gamma_{rdR}(T_2)\geq|V(T_2)|+1$ by Corollary \ref{cor-tree}. Let $w\in V(T_1)$. Since $T[V_0]$ has no isolated vertices, $w$ has a neighbor $w'\in V_0$. On the other hand, it is easy to see that $f\mid_{V(T_1)}$ and $f\mid_{V(T_2)}$ are RDRD functions of $T_1$ and $T_2$, respectively. This shows that
\begin{equation*}
\begin{array}{lcl} 
n+2&=&|V(T_1)|+|V(T_2)|+2\leq \gamma_{rdR}(T_1)+\gamma_{rdR}(T_2)\\
&\leq&f\big{(}V(T_1)\big{)}+f\big{(}V(T_2)\big{)}=f\big{(}V(T)\big{)}=\gamma_{rdR}(T)=n+2.
\end{array}
\end{equation*}
This implies that $f\big{(}V(T_1)\big{)}=\gamma_{rdR}(T_1)=|V(T_1)|+1$ and $f\big{(}V(T_2)\big{)}=\gamma_{rdR}(T_2)=|V(T_2)|+1$. In particular, this shows that $T_1$ is isomorphic to a star on at least three vertices with central vertex $w$. This contradicts the fact that $w'\in V_0$ is a leaf of $T$. So, there is no vertex in $V_0$ having a neighbor in $V_3$ as well as a neighbor in $v\in V_2\cup V_3$. 
 
Suppose that $v\in V_0$ is adjacent to at least three vertices $x,y,z\in V_2$. Let $T_1$ and $T_2$ be the two components of $T-vx$ such that $x\in V(T_1)$. Note that $v$ is adjacent to a vertex $v'\in V_0$. If $T_2$ is a star, then $v'\in V_0$ is a leaf of $T$, which is impossible. Therefore, $T_2$ is different from a star. Thus, $\gamma_{rdR}(T)=f\big{(}V(T_1)\big{)}+f\big{(}V(T_2)\big{)}\geq \gamma_{rdR}(T_1)+\gamma_{rdR}(T_2)\geq|V(T_1)|+1+|V(T_1)|+2=n+3$. This is a contradiction.

The above discussion shows that every vertex in $V_0$ is adjacent to ``only one vertex in $V_3$" or ``precisely two vertices in $V_2$".\vspace{0.75mm}

Let $H$ be a component of $T[V_{0}]$. We claim that $|V(H)|=2$. Let $|V(H)|=k_{1}+k_{2}$ in which $k_{1}$ is the number of vertices of $H$ having one neighbor in $V_3$, and $k_2$ is the number of vertices of $H$ having two neighbors in $V_2$. Since $T$ is a tree, $T-V(H)$ has exactly $k_{1}+2k_{2}$ components $T_{1},\cdots,T_{k_1},T_{k_{1}+1},\cdots,T_{k_{1}+2k_{2}}$. Moreover, we let $r=r_{1}+r_{2}$ in which $r_1$ and $r_2$ are the number of components among $\{T_{1},\cdots,T_{k_1}\}$ and $\{T_{k_{1}+1},\cdots,T_{k_{1}+2k_{2}}\}$ which are stars or isolated vertices, respectively. Notice that $\gamma_{rdR}(T_i)\geq|V(T_i)|+2$ for those $T_i$ different from stars and isolated vertices by Theorem \ref{Tree2}, and $\gamma_{rdR}(T_i)=|V(T_i)|+1$ for any other $T_i$ by Corollary \ref{cor-tree}. Therefore, 
\begin{equation}\label{lower}
\begin{array}{lcl}  
\sum_{i=1}^{k_{1}+2k_{2}}\gamma_{rdR}(T_i)&\geq&\sum_{i=1}^{k_{1}+2k_{2}}|V(T_i)|+r+2(k_{1}+2k_{2}-r)\\
&=&n-(k_{1}+k_{2})+r+2(k_{1}+2k_{2}-r)=n+k_{1}+3k_{2}-r.
\end{array}
\end{equation}

Without loss of generality, we may assume that $T_{1},\cdots,T_{r_1}$ are the only components among $\{T_{1},\cdots,T_{k_1}\}$ which are stars or isolated vertices. Note that $f\mid_{V(T_i)}$ is an RDRD function of $T_i$ for each index $i$. In particular, it follows that $f\big{(}V(T_i)\big{)}\geq \gamma_{rdR}(T_i)$ for each $r_{1}+1\leq i\leq k_{1}+2k_{2}$. We now consider a component $T_i$ for some $1\leq i\leq r_{1}$. Since $T_i$ is a star or an isolated vertex, it can be checked that $f\big{(}V(T_i)\big{)}\geq \gamma_{rdR}(T_i)+1$. Therefore,
\begin{equation}\label{upper} 
\sum_{i=1}^{k_{1}+2k_{2}}\gamma_{rdR}(T_i)+r_{1}\leq \sum_{i=1}^{k_{1}+2k_{2}}f\big{(}V(T_i)\big{)}=\gamma_{rdR}(T)=n+2.
\end{equation}
Together inequalities (\ref{lower}) and (\ref{upper}) by taking $r=r_{1}+r_{2}$ into account show that
\begin{equation}\label{con} 
n+k_{1}+3k_{2}-r_{2}\leq \sum_{i=1}^{k_{1}+2k_{2}}\gamma_{rdR}(T_i)+r_{1}\leq n+2.
\end{equation}
On the other hand, $2k_{2}-r_{2}\geq0$. So, we have $n+k_{1}+k_{2}\leq n+2$. Therefore, $k_{1}+k_{2}\leq2$. This necessarily implies that $|V(H)|=k_{1}+k_{2}=2$, as claimed. This means that every component of $T[V_0]$ is isomorphic to $K_2$.

In what follows, we show that $T[V_{2}\cup V_{3}]$ is edgeless. Suppose to the contrary that there exist two adjacent vertices $x,y\in V_{2}\cup V_{3}$. By removing the edge $xy$ of $T$, we get two components $T_1$ and $T_2$. If at least one of them, say $T_1$, is neither a star nor an isolated vertex, then $\gamma_{rdR}(T_1)\geq|V(T_1)|+2$. So, $\gamma_{rdR}(T)=f\big{(}V(T_1)\big{)}+f\big{(}V(T_2)\big{)}\geq \gamma_{rdR}(T_1)+\gamma_{rdR}(T_2)\geq n+3$, which is a contradiction. So, we assume that both $T_1$ and $T_2$ are stars or one of them is a star and the other one is an isolated vertex. In such a situation, it turns out that $T$ is of the form one of the trees depicted in Figure \ref{fig33}. 

\begin{figure}[h]
\tikzstyle{every node}=[circle, draw, fill=white!, inner sep=0pt,minimum width=.16cm]
\begin{center}
\begin{tikzpicture}[thick,scale=.6]
  \draw(0,0) { % <-- START CO-ORDINATES

+(-10,0) node{}
+(-7,0) node{}
+(-10,0) --+(-7,0)
+(-10,0.4) node[rectangle, draw=white!0, fill=white!100]{$x$}
+(-7,0.4) node[rectangle, draw=white!0, fill=white!100]{$y$}
%+(-10.75,-1.4) node[rectangle, draw=white!0, fill=white!100]{$u$}
%+(-8.5,-2) node[rectangle, draw=white!0, fill=white!100]{$T_1$}

+(-10.75,-1) node{}
+(-9.25,-1) node{}
+(-10.25,-1) node[rectangle, draw=white!0, fill=white!100]{$.$}
+(-10,-1) node[rectangle, draw=white!0, fill=white!100]{$.$}
+(-9.75,-1) node[rectangle, draw=white!0, fill=white!100]{$.$}

+(-7.75,-1) node{}
+(-6.25,-1) node{}
+(-7.25,-1) node[rectangle, draw=white!0, fill=white!100]{$.$}
+(-7,-1) node[rectangle, draw=white!0, fill=white!100]{$.$}
+(-6.75,-1) node[rectangle, draw=white!0, fill=white!100]{$.$}

+(-10.75,-1) --+(-10,0) --+(-9.25,-1)
+(-7.75,-1) --+(-7,0) --+(-6.25,-1)

%%%%%%%%%%%%%%%%%%%%%%%%%%%%%%%%%%%%%%%%%%%%%%%%%%%%%%%%%%%%%%%%%%%%%%%%%%%%%%%

+(-4,0) node{}
+(-1,0) node{}
+(-4,0.4) node[rectangle, draw=white!0, fill=white!100]{$x$}
+(-1.75,-1.4) node[rectangle, draw=white!0, fill=white!100]{$y$}
%+(-4.75,-1.4) node[rectangle, draw=white!0, fill=white!100]{$u$}
%+(-2.5,-2) node[rectangle, draw=white!0, fill=white!100]{$T_2$}

+(-4.75,-1) node{}
+(-3.25,-1) node{}
+(-4.25,-1) node[rectangle, draw=white!0, fill=white!100]{$.$}
+(-4,-1) node[rectangle, draw=white!0, fill=white!100]{$.$}
+(-3.75,-1) node[rectangle, draw=white!0, fill=white!100]{$.$}

+(-1.75,-1) node{}
+(-0.25,-1) node{}
+(-1.25,-1) node[rectangle, draw=white!0, fill=white!100]{$.$}
+(-1,-1) node[rectangle, draw=white!0, fill=white!100]{$.$}
+(-0.75,-1) node[rectangle, draw=white!0, fill=white!100]{$.$}

+(-4.75,-1) --+(-4,0) --+(-3.25,-1)
+(-1.75,-1) --+(-1,0) --+(-0.25,-1)
+(-4,0) --+(-1.75,-1)

%%%%%%%%%%%%%%%%%%%%%%%%%%%%%%%%%%%%%%%%%%%%%%%%%%%%%%%%%%%%%%%%%%%%%%%%%%%%%%%%%

+(2,0) node{}
+(5,0) node{}
+(2.75,-1.4) node[rectangle, draw=white!0, fill=white!100]{$x$}
+(4.25,-1.4) node[rectangle, draw=white!0, fill=white!100]{$y$}
%+(2,0.4) node[rectangle, draw=white!0, fill=white!100]{$u$}
%+(3.5,-2) node[rectangle, draw=white!0, fill=white!100]{$T_3$}

+(1.25,-1) node{}
+(2.75,-1) node{}
+(1.75,-1) node[rectangle, draw=white!0, fill=white!100]{$.$}
+(2,-1) node[rectangle, draw=white!0, fill=white!100]{$.$}
+(2.25,-1) node[rectangle, draw=white!0, fill=white!100]{$.$}

+(4.25,-1) node{}
+(5.75,-1) node{}
+(4.75,-1) node[rectangle, draw=white!0, fill=white!100]{$.$}
+(5,-1) node[rectangle, draw=white!0, fill=white!100]{$.$}
+(5.25,-1) node[rectangle, draw=white!0, fill=white!100]{$.$}

+(1.25,-1) --+(2,0) --+(2.75,-1)
+(4.25,-1) --+(5,0) --+(5.75,-1)
+(2.75,-1) --+(4.25,-1)

%%%%%%%%%%%%%%%%%%%%%%%%%%%%%%%%%%%%%%%%%%%%%%%%%%%%%%%%%%%%%%%%%%%%%%%%%%%%%%%%%%

+(8,0) node{}
%+(8,0.4) node[rectangle, draw=white!0, fill=white!100]{$u$}
%+(9.5,-2) node[rectangle, draw=white!0, fill=white!100]{$T_4$}

+(7.25,-1) node{}
+(8.75,-1) node{}
+(7.75,-1) node[rectangle, draw=white!0, fill=white!100]{$.$}
+(8,-1) node[rectangle, draw=white!0, fill=white!100]{$.$}
+(8.25,-1) node[rectangle, draw=white!0, fill=white!100]{$.$}
+(9,-0.6) node[rectangle, draw=white!0, fill=white!100]{$x$}

+(10.25,-1) node{}
+(10.25,-0.6) node[rectangle, draw=white!0, fill=white!100]{$y$}

+(7.25,-1) --+(8,0) --+(8.75,-1)
+(8.75,-1) --+(10.25,-1)

};
\end{tikzpicture}
\end{center}
\caption{In each tree, $f(x)=f(y)\in \{2,3\}$.}\label{fig33}
\end{figure}
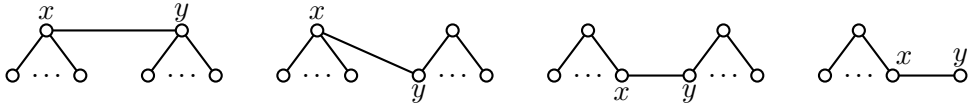 
In each tree depicted in Figure \ref{fig33}, no vertex is labeled $0$ under $f$. But this contradicts our assumption that $V_0$ is nonempty. Therefore, $T[V_{2}\cup V_{3}]$ is edgeless.

Finally, we claim that there is no vertex $v\in V_2$ such that $|N(v)\cap V_0|\geq2$. Suppose to the contrary that there exists a vertex $v\in V_2$ with two neighbors $u,w\in V_0$. Note that $u$ is adjacent to a vertex $x\in V_{2}\setminus\{v\}$ as well as a vertex $u'\in V_{0}\setminus\{w\}$. Let $T_1$ and $T_2$ be the components of $T-uv$ such that $u\in V(T_{1})$. We observe that $T_1$ is different from a star, for otherwise $u'$ would be a leaf with $f(u')=0$, which is impossible. Also, $T_2$ is different from a star by a similar fashion. Therefore, $\gamma_{rdR}(T_1)\geq|V(T_1)|+2$ and $\gamma_{rdR}(T_2)\geq|V(T_2)|+2$. On the other hand, we replace the weight $2$ of $x$ with $3$. This gives us an RDRD function $f'$ of $T$ of weight $\gamma_{rdR}(T)+1$. Moreover, $f'\mid_{V(T_i)}$ is an RDRD function of $T_i$ for each $i\in \{1,2\}$. We now deduce that
$$\gamma_{rdR}(T)+1=f'\big{(}V(T_1)\big{)}+f'\big{(}V(T_2)\big{)}\geq \gamma_{rdR}(T_1)+\gamma_{rdR}(T_2)\geq|V(T_1)|+|V(T_2)|+4,$$
and so $\gamma_{rdR}(T)\geq n+3$. This is a contradiction. Consequently, every vertex in $V_2$ is adjacent to exactly one vertex in $V_0$.
 
In summary, we have the following statements:\\
$\bullet$ every component of $T[V_0]$ is isomorphic to $K_2$,\\
$\bullet$ every vertex in $v\in V_0$ is adjacent to ``only one vertex in $V_3$ and that $N(v)\cap V_{2}=\emptyset$" or ``it has precisely two neighbors in $V_2$ and that $N(v)\cap V_{3}=\emptyset$",\\
$\bullet$ each vertex in $V_2$ is adjacent to precisely one vertex in $V_0$,\\
$\bullet$ every vertex in $V_3$ has at least one neighbor in $V_0$, and\\
$\bullet$ $T[V_{2}\cup V_{3}]$ is edgeless.

Taking the above statements into consideration and that every vertex in $V_1$ must have a neighbor in $V_{2}\cup V_{3}$, we conclude that $T\in \mathcal{T}_1\cup \mathcal{T}_2$.\vspace{0.75mm}

%Finally, we show that $T[V_0]$ has only one component. Suppose that $T[V_0]$ has at least two components and $H$ is one of them. Note that $T-V(H)$ has precisely four components $T_i$ for $1\leq i\leq4$. Since $T[V_0]$ has more than one component, it is easily observed that at least one of the components $T_i$ is neither a star nor an isolated vertex. Therefore, 
%$$\gamma_{rdR}(T)=\sum_{i=1}^{4}f\big{(}V(T_i)\big{)}\geq \sum_{i=1}^{4}\gamma_{rdR}(T_i)\geq \sum_{i=1}^{4}|V(T_i)|+5=n+3.$$
%This is a contradiction. So, $T[V_0]$ has only one component. Consequently, $T\in \mathcal{T}_1$.\vspace{0.75mm}

%\textit{Case 3.} $V_2,V_3\neq \emptyset$. Analogously to the way we argued in Case $1$ and Case $2$, we can show that every component of $T[V_0]$ is of order $2$, $T[V_3]$ is edgeless, every vertex of $V_0$ has at most one neighbor in $V_3$, and each vertex of $V_0$ cannot have neighbors in both $V_2$ and $V_3$.\\

Conversely, let $T'\in \mathcal{T}_1\cup \mathcal{T}_2$. It is routine to see that $\gamma_{rdR}(T')=n+2$ if $T'\in \mathcal{T}_1$. Now, let $T'\in \mathcal{T}_2$. Then, $T'$ is obtained from $T$ as given in the description of $\mathcal{T}_2$. Let $L_{s}(T)$ be the set of leaves of $T$ whose support vertices are strong. Now, the assignment $2$ to all vertices in $L_{s}(T)$, $3$ to all vertices in $L(T)\setminus L_{s}(T)$ as well as to each non-leaf vertex $y\in V(T)$ with $d_{T}(x,y)\equiv0\ (\mbox{mod}\ 3)$ for a given $x\in L(T)$, $0$ to any other vertex of $T$, and $1$ to the other vertices of $T'$ gives us an RDRD function of $T'$ with weight $|V(T')|+2$, and so $\gamma_{rdR}(T')\leq|V(T')|+2$. Therefore, $\gamma_{rdR}(T')=|V(T')|+2$. This completes the proof.
\end{proof}

%%%%%%%%%%%%%%%%%%%%%%%%%%%%%%%%%%%%%%%%%%%%%%%%%%%%%%%%%%%%%%%%%%%%%%%%%%%%%

\section{Graphs with small RDRD numbers}

We give the characterizations of connected graphs $G$ with RDRD numbers $\gamma_{rdR}(G)\in\{3,4,5\}$ in this section. This is simple when we turn our attention to connected graphs with RDRD number equals three. So, we only state the following theorem and omit the proof.

Recall first that by starting with a disjoint union of two graphs $G$ and $H$ and adding edges joining every vertex of $G$ to every vertex of $H$, we obtain the {\em join} of $G$ and $H$, denoted $G\vee H$.

\begin{theorem}\label{three}
For any connected graph $G$, $\gamma_{rdR}(G)=3$ if and only if $G=K_{2}$ or $G=K_{1}\vee H$ in which $H$ is any graph with no isolated vertices.
\end{theorem}

\begin{theorem}\label{four}
For any connected graph $G$, $\gamma_{rdR}(G)=4$ if and only if $G\in \Theta=\{\overline{K_{2}}\vee H_{1},K_{1}\vee(K_{1}+H_{2})\}\cup\{P_{3}\}$ in which $H_{1}$ and $H_{2}$ are any graphs with no isolated vertices.
\end{theorem}
\begin{proof}
It is readily seen that $\gamma_{rdR}(G)=4$ for all $G\in \Theta$. Conversely, we suppose that $\gamma_{rdR}(G)=4$ such that $G\neq P_{3}$. Let $f=(V_{0},V_{1},V_{2},V_{3})$ be a $\gamma_{rdR}(G)$-function. We distinguish two cases depending of the possible values for $|V_{1}|$, $|V_{2}|$ and $|V_{3}|$.\vspace{0.25mm}

\textit{Case 1.} $|V_{2}|=2$. Let $V_{2}=\{x,y\}$. This implies that $V_{1}=V_{3}=\emptyset$ and $V_{0}\neq \emptyset$. By the definition, $H_{1}=G[V_{0}]$ has no isolated vertices. Moreover, all vertices of $H_{1}$ are adjacent to both $x$ and $y$. If $xy\in E(G)$, then the assignment $g(x)=3$ and $g(v)=0$ for the other vertices $v$ is an RDRD function of $G$ with $\omega(g)=3$, which is impossible. Therefore, $xy\notin E(G)$. This shows that $G\cong \overline{K_{2}}\vee H_{1}$.\vspace{0.25mm}

\textit{Case 2.} $(|V_{1}|,|V_{3}|)=(1,1)$. In such a situation, $V_{0}\neq \emptyset$ and $V_{2}=\emptyset$. Moreover, $H_{2}=G[V_{0}]$ has no isolated vertices. Let $V_{1}=\{y\}$ and $V_{3}=\{x\}$. We have $\mbox{deg}(x)=|V(G)|-1$ by the definition. Note that $y$ has no neighbor in $V_{0}$, for otherwise $g(x)=3$ and $g(v)=0$ for any other vertex $v$ would be an RDRD function of $G$ of weight $3$, which is a contradiction. Therefore, $G\cong K_{1}\vee(K_{1}+H_{2})$.\vspace{0.25mm}

In both cases, we have $G\in \Theta$.
\end{proof}

In order to characterize all connected graphs with RDRD number five, we introduce the family $\Omega$ as the union of the following subfamilies:\vspace{0.75mm}\\
$\Omega_{1}$: All graphs $G$ obtained from any graph $H$ with no isolated vertices by adding three new vertices $x$, $y$ and $z$ such that both $x$ and $y$ are adjacent to all vertices of $H$ and $xz\in E(G)$.\vspace{0.75mm}\\
$\Omega_{2}$: The family of graphs $G'$ obtained from any graph $G\in \Omega_{1}$ by joining $z$ to some non-universal vertices (if any) of $H$.\vspace{0.75mm}\\
$\Omega_{3}$: The family of graphs $G''$ obtained from any graph $G\in \Omega_{1}$ by joining $z$ to $y$.\vspace{0.75mm}\\
$\Omega_{4}$: All graphs $G$ obtained from any graph $H$ with no isolated vertices by adding three new vertices $x$, $a$ and $b$ such that $\mbox{deg}(x)=|V(G)|-1$ and $a$ and $b$ have degree one.\vspace{0.75mm}\\
$\Omega_{5}$: All graphs $G$ obtained from any graph $H$ with no isolated vertices by adding two new vertices $x$ and $y$, joining $x$ to all vertices of $H$ and joining $y$ to $1\leq i\leq|V(H)|-1$ vertices of $H$.

\begin{theorem}\label{five}
For any connected graph $G$ of order $n$, $\gamma_{rdR}(G)=5$ if and only if $G\in \Omega\cup\{K_{1,3}\}$.
\end{theorem}
\begin{proof}
Suppose that $f=(V_{0},V_{1},V_{2},V_{3})$ is an RDRD function of $G$ with weight $\omega(f)=\gamma_{rdR}(G)=5$. We consider the following cases.\vspace{0.25mm}

\textit{Case 1.} $(|V_{1}|,|V_{2}|)=(3,1)$. It is then easy to see that $V_{0}=V_{3}=\emptyset$, necessarily. Let $V_{1}=\{a,b,c\}$ and $V_{2}=\{x\}$. Then, $x$ must be adjacent to all vertices in $V_{1}$. If there exists an edge among the vertices of $V_{1}$, say $ab$, then $\big{(}g(a),g(b),g(c),g(x)\big{)}=(0,0,1,3)$ is an RDRD function of $G$ with weight four, a contradiction. Therefore, $\{a,b,c\}$ is independent. So, $G\cong K_{1,3}$.\vspace{0.25mm}

\textit{Case 2.} $(|V_{1}|,|V_{2}|)=(1,2)$. Let $V_{1}=\{z\}$ and $V_{2}=\{x,y\}$. If $V_{0}=\emptyset$, then $G\cong P_{3}$ or $C_{3}$. This is impossible as $\gamma_{rdR}(P_{3})=4$ and $\gamma_{rdR}(C_{3})=3$. So, we must have $V_{0}\neq\emptyset$. By the definition, $H=G[V_{0}]$ is a graph with no isolated vertices whose all vertices are adjacent to both $x$ and $y$. Moreover, $z$ has at least one neighbor in $\{x,y\}$, say $x$. If $xy\in E(G)$, then $\big{(}g(x),g(y)\big{)}=(3,0)$ and $g(v)=f(v)$ for the other vertices $v$ is an RDRD function of $G$ with $\omega(g)=4$, a contradiction. Therefore, $xy\notin E(G)$. We clearly observe that $G\in \Omega_{1}$ if $\mbox{deg}(z)=1$. So, we may assume that $\mbox{deg}(z)\geq2$. We distinguish two possibilities depending on the behavior of $z$.\vspace{0.25mm}

$\bullet$ $z$ has some neighbors in $V_{0}$. Suppose that $z$ is adjacent to a universal vertex $z'$ of $H$. Then, we get the RDRD function $g(z')=3$ and $g(v)=0$ for any other vertex $v$ of weight $3$. This is a contradiction. This shows that $z$ is not adjacent to any universal vertex of $H$. So, $G\in \Omega_{2}$.\vspace{0.25mm}

$\bullet$ $z$ has no neighbor in $V_{0}$. Then, $z$ is necessarily adjacent to $y$. Therefore, $G\in \Omega_{3}$.\vspace{0.25mm}

\textit{Case 3.} $(|V_{1}|,|V_{3}|)=(2,1)$. We easily observe that $V_{2}=\emptyset$ and $V_{0}\neq \emptyset$. Moreover, $H=G[V_{0}]$ has no isolated vertices by the definition. Let $V_{1}=\{a,b\}$ and $V_{3}=\{x\}$. Note that $x$ is adjacent to any other vertex of $G$. If $ab\in E(G)$, then $g(a)=g(b)=0$ and $g(v)=f(v)$ for the other vertices $v$ gives us an RDRD function of $G$ of weight three, which is impossible. Therefore, $ab\notin E(G)$. Suppose now that at least one of $a$ and $b$, say $a$, has a neighbor in $V_{0}$. Then, $g(a)=0$ and $g(v)=f(v)$ for any other vertex $v$ gives an RDRD function of $G$ with $\omega(g)=4$. This is impossible. Therefore, both $a$ and $b$ have no neighbor in $V_{0}$. In such a situation, we observe that $G$ belongs to $\Omega_{4}$.\vspace{0.25mm}

\textit{Case 4.} $(|V_{2}|,|V_{3}|)=(1,1)$. Clearly, $V_{1}=\emptyset$ and $V_{0}\neq \emptyset$. By the definition, $H=G[V_{0}]$ is a graph with no isolated vertices. Let $V_{2}=\{y\}$ and $V_{3}=\{x\}$. In order to fulfill the conditions of the RDRD function $f$ of $G$, all vertices of $H$ must be adjacent to $x$. If $xy\in E(G)$, then $\big{(}g(x),g(y)\big{)}=(3,1)$ and $g(v)=0$ for the other vertices $v$ is an RDRD function of $G$ with weight less than five. Therefore, $xy\notin E(G)$. If $N(y)\cap V(H)=\emptyset$, then $y$ is adjacent to $x$ since $G$ is connected, a contradiction. Finally, if $N(y)=V(H)$, then the assignment $g(x)=g(y)=2$ and $g(v)=0$ for all $v\in V(G)\setminus\{x,y\}$ results in $\gamma_{rdR}(G)=4$. This is impossible. Thus, $N(x)=V(H)$ and $y$ is adjacent to $1\leq i\leq|V(H)|-1$ vertices of $H$. So, $G\in \Omega_{5}$.\vspace{0.25mm}

All in all, we have concluded that $G\in \Omega$.

Conversely, it is readily seen that $\gamma_{rdR}(G)=5$ for all $G\in \Omega\cup\{K_{1,3}\}$. This completes the proof.
\end{proof}

%%%%%%%%%%%%%%%%%%%%%%%%%%%%%%%%%%%%%%%%%%%%%%%%%%%%%%%%%%%%%%%%%%%%%%%%%%%%

\vspace{3mm}
B. Samadi and N. Soltankhah\\
Department of Mathematics, Faculty of Mathematical Sciences, Alzahra University,\\ Tehran, Iran\\
{\it b.samadi@alzahra.ac.ir}\ \ \ {\it soltan@alzahra.ac.ir}\\[0.3cm]
H. Abdollahzadeh Ahangar\\
Department of Mathematics, Babol Noshirvani University of Technology,\\ Shariati Ave., Babol, I.R. Iran, Postal Code: 47148-71167\\
{\it ha.ahangar@nit.ac.ir}\\[0.3cm]
M. Chellali\\
LAMDA-RO Laboratory, Department of Mathematics, University of Blida,\\ Blida, Algeria\\
{\it m\_chellali@yahoo.com}\\[0.3cm]
D.A. Mojdeh\\
Department of Mathematics, University of Mazandaran,\\ Babolsar, Iran\\
{\it damojdeh@umz.ac.ir}\\[0.3cm]
S.M. Sheikholeslami\\
Department of Mathematics, Azarbaijan Shahid Madani University,\\ Tabriz, Iran\\
{\it s.m.sheikholeslami@azaruniv.ac.ir}\\[0.3cm]
J.C. Valenzuela-Tripodoro\\
Department of Mathematics, University of C\'adiz, Spain\\
{\it jcarlos.valenzuela@uca.es}

\end{document}